\documentclass[12pt]{amsart}
\usepackage[utf8]{inputenc}
\usepackage{amsmath,amssymb}
\usepackage{mathtools}
\usepackage{subcaption}
\usepackage{microtype}
\usepackage{import}
\usepackage{svg}
\usepackage{lmodern}
\usepackage[colorinlistoftodos]{todonotes}
\usepackage{amsthm}
\usepackage{amsrefs}
\usepackage{thmtools}
\usepackage{comment}

\usepackage[left=1in,
            right=1in,
            top=1in,
            bottom=1in, footskip=.25in]
            {geometry}

\newcommand{\mc}[1]{\ensuremath{\mathcal{#1}}}

\newcommand{\bR}{\ensuremath{\mathbb{R}}}

\newcommand{\cF}{\ensuremath{\mathcal{F}}}

\newcommand{\mute}[1]{}

\DeclareMathOperator{\Mod}{Map}
\DeclareMathOperator{\Aut}{Aut}
\DeclareMathOperator{\Ends}{Ends}

\newcommand{\EMod}{\ensuremath{\Mod^\pm}}
\newcommand{\Id}{\ensuremath{\mathrm{Id}}}

\newtheorem{theorem}{Theorem}[section]

\newtheorem{conjecture}[theorem]{Conjecture}

\newtheorem{lemma}[theorem]{Lemma}
\newtheorem{corollary}[theorem]{Corollary}

\theoremstyle{definition}
\newtheorem{definition}[theorem]{Definition}
\newtheorem{question}[theorem]{Question}

\theoremstyle{remark}
\newtheorem*{remark}{Remark}

\title{Big Flip Graphs and their Automorphism Groups}
\author[A. Bar-Natan, A. Goel, B. Halstead, P. Hamrick, S. Shenoy, R. Verma]{Assaf Bar-Natan, Advay Goel, Brendan Halstead, \\ 
Paul Hamrick, Sumedh Shenoy, and Rishi Verma}

\date{}

\begin{document}
\maketitle

\begin{abstract}
In this paper, we study the relationship between the mapping class group of an infinite-type surface and the simultaneous flip graph, a variant of the flip graph for infinite-type surfaces defined by Fossas and Parlier \cite{fossas-parlier}. We show that the extended mapping class group is isomorphic to a proper subgroup of the automorphism group of the flip graph, unlike in the finite-type case. This shows that Ivanov's metaconjecture, which states that any ``sufficiently rich" object associated to a finite-type surface has the extended mapping class group as its automorphism group, does not extend to simultaneous flip graphs of infinite-type surfaces.
\end{abstract}

\section{Introduction}
The \textit{mapping class group} $\Mod(S)$ of a surface $S$ is the group of orientation-preserving homeomorphisms of $S$ up to isotopy. Expanding $\Mod(S)$ to include orientation-reversing homeomorphisms yields the \emph{extended mapping class group} $\EMod(S)$. The \textit{flip graph} $\cF(S)$ of $S$ is the graph whose vertices are triangulations and whose edges are elementary moves, or \emph{flips}. This graph characterizes finite-type surfaces up to homeomorphism, and the automorphisms of the flip graph are precisely those induced by the action of the extended mapping class group \cite{korkmaz-popadopoulos}. This exemplifies a broader pattern described by the following ``metaconjecture" due to Ivanov.

\begin{conjecture}[Problem 6 of \cite{ivanov}]
Every object naturally associated to a surface and having sufficiently rich structure has the extended mapping class group as its group of automorphisms.
\end{conjecture}

Here, ``sufficiently rich" is intentionally left undefined. Ivanov's metaconjecture has been shown to hold for many objects in addition to the flip graph, such as the curve complex \cite{ivanov-curve-complex}, the arc complex \cite{irmak-mccarthy}, and Torelli buildings \cite{farb-ivanov}, and in 2019, the metaconjecture was shown to be true for an even larger class of complexes on finite-type surfaces \cite{brendle-margalit}.

For infinite-type surfaces, i.e., surfaces whose fundamental group is not finitely-generated, it is known that the automorphism group of the curve graph is the extended mapping class group \cites{bavard-dowdall-rafi, hernandez-morales-valdez}. Therefore, in this paper, motivated by Ivanov's metaconjecture and the above results, we study the same question in the setting of flip graphs on infinite-type surfaces. 
We will consider the simultaneous flip graph, a variant of the flip graph defined by Fossas and Parlier \cite{fossas-parlier}, which adapts the flip graph for infinite-type surfaces. 
We show that---unlike in the finite-type case---the symmetries of this graph are no longer fully captured by the extended mapping class group. In particular, we prove the following result.

\begin{theorem}\label{thm:isomorphic-to-subgroup}
The natural action of the extended mapping class group on the flip graph is an isomorphism from $\EMod(\Sigma)$ to a proper subgroup of $\Aut(\cF(\Sigma))$.
\end{theorem}

To prove this, we first show that the isotopy class of a homeomorphism is determined by its action on $\cF(\Sigma)$. For other simplicial complexes on which $\EMod(\Sigma)$ acts, such as the curve complex and the pants complex, this can be accomplished using the Alexander Method for infinite-type surfaces developed by Hernandez et al. in \cite{alex-ref-pls}. To apply this result to ideal arcs, we prove the following lemma, which may be of interest by its own merit.
\begin{restatable*}[Alexander Method for Ideal Arcs]{lemma}{alex}
\label{lemma:fix-arc-fix-curve-alex-ref-pls}
Let $\Sigma$ be an infinite-type surface, and let $f$ be a homeomorphism of $\Sigma$. For $e \in \Ends(\Sigma)$, let $\mathcal{L}_e$ denote the collection of ideal arcs that are ``loops at $e$" --- those which map both ends of $(0, 1)$ to $e$. The following are equivalent.
\begin{enumerate}
    \item[(i)] There exists $e \in \Ends(\Sigma)$ such that for all $\lambda \in \mathcal{L}_e$, the ideal arc $f(\lambda)$ is isotopic to $\lambda$.
    \item[(ii)] The homeomorphism $f$ is isotopic to the identity homeomorphism $\Id_\Sigma$.
\end{enumerate}
\end{restatable*}

Using the terminology of \cite{alex-ref-pls}, this means that for any end $e$, the set $\mathcal{L}_e$ forms an Alexander system for $\Sigma$. This also establishes that the mapping class group acts faithfully on the arc graphs discussed in \cite{aramayona-fossas-parlier}.

\subsection*{Acknowledgements}
We would like to thank the Canada/USA Mathcamp for 
providing a welcoming space and setting to meet, discuss, and develop the ideas in this paper. We would also like to thank Kasra Rafi, Hugo Parlier, and Alan McLeay for their support and helpful conversations. Finally, we would like to thank Mark Bell for telling us about the Python package ``Bigger'' and talking to us about the subtleties of triangulations on infinite-type surfaces.
\section{Preliminaries}

This section is an introduction to the concepts and terminology regarding infinite-type surfaces that will be used throughout the paper.
A \emph{surface} is a connected, orientable topological 2-manifold. A surface $\Sigma$ is called \emph{finite-type} if $\pi_1(\Sigma)$ is finitely-generated. Otherwise, $\Sigma$ is called \emph{infinite-type}.

The \emph{end space} of an infinite-type surface $\Sigma$, introduced by Freudenthal \cite{freudenthal}, and denoted by $\Ends(\Sigma)$, is $\varprojlim \pi_0(\Sigma \setminus K)$, where $K$ ranges over all compact subsurfaces of $\Sigma$ and the discrete spaces $\pi_0(\Sigma \setminus K)$ are ordered by inclusion. We equip $\Ends(\Sigma)$ with the inverse limit topology, and note that it is totally disconnected, compact, and metrizable. The 
set $\bar\Sigma := \Sigma \sqcup \Ends(\Sigma)$, equipped with an appropriate topology, is called the \textit{Freudenthal compactification} of the surface $\Sigma$ \cite{freudenthal}. Concretely, an end can be thought of as a descending chain $e = U_1 \supset U_2 \supset \cdots$ of connected open sets such that
$\cap_i U_i = \emptyset$. If there exists $n$ such that $U_n$ can be embedded in $\bR^2$, then $e$ is called \emph{planar}. Otherwise, it is said to be \emph{non-planar} or \emph{accumulated by genus}. We denote by $\Ends_g(\Sigma)$ the subspace of non-planar ends. See \cite{fanoni-ghaswala-mcleay} for a more detailed exposition. 

In 1963, Richards showed that orientable infinite-type surfaces are classified up to homeomorphism by their genus and the pair of spaces $(\Ends(\Sigma), \Ends_g(\Sigma))$. For more information, see \cite{richards}. 

A \emph{simple closed curve} is an embedding of the circle $S^1$ into $\Sigma$. An \emph{ideal arc}, or simply an \emph{arc}, on a surface $\Sigma$ is an embedding $\alpha:(0,1) \to \Sigma$ which extends continuously to a map $\bar\alpha: [0,1] \to \bar\Sigma$  such that the ends of $(0,1)$ are mapped to ends of $\Sigma$. A \emph{multi-arc} is a collection of disjoint arcs.
An arc $\alpha$ in $\Sigma$ is called \emph{peripheral} if $\Sigma \setminus \alpha$ has a connected component that is a disk. An arc is called \emph{essential} if it is not 
peripheral. A \textit{multi-arc} 
is a possibly-infinite collection of disjoint arcs, and an isotopy class of a multi-arc is the collection of isotopy classes of its corresponding disjoint arcs.

Let $[\mu]$ and $[\nu]$ be isotopy classes of multi-arcs. The \emph{intersection number} 
$i(\mu, \nu)$ is the minimum value of $|\mu \cap \nu|$ among all isotopy representatives of $\mu$ and $\nu$.

\begin{definition}
An \emph{ideal triangulation} of a surface, denoted $T$, is a locally-finite maximal collection pairwise disjoint, pairwise non-homotopic, essential ideal arcs whose complementary components are open disks. We also require that $T$ have a tubular neighbourhood homotopic to itself.
\end{definition}

There are other alternative definitions of ideal triangulations on infinite-type surfaces in the literature, such as in \cites{fossas-parlier, mcleay-parlier}. We will sometimes consider a triangulation of $\Sigma$ to be a subspace of $\Sigma$ rather than a collection of maps into $\Sigma$. The closures of complementary components of a triangulation are called $\emph{triangles}$. It follows from maximality that the boundary of a triangle can consist of either two or three arcs. 

If an arc $\alpha$ in a triangulation bounds only one triangle $\Delta$, then $\partial \Delta$ must consist of a single additional arc $\beta$ \cite{korkmaz-popadopoulos}. In this situation we call $\beta$ a \emph{petal arc for $\alpha$}. 

In \cite{mcleay-parlier}, McLeay and Parlier establish that every infinite-type surface admits an ideal triangulation. Moreover, they determine when exactly a multiarc is a subset of a triangulation:

\begin{lemma}[Theorem 1.2 of \cite{mcleay-parlier}]\label{lemma:complete-multiarc-to-tri}
An ideal multiarc $\mu$ is a subset of an ideal triangulation if and only if for any simple closed curve $\gamma$, the intersection number $i(\gamma, \mu) < \infty$.
\end{lemma}

If a multiarc $\mu$ satisfies the conditions above, then there exists a triangulation $T$ containing $\mu$. We call such a triangulation a \emph{completion} of $\mu$. 

\begin{definition}
Consider a triangulation $T$ and a disjoint collection of arcs $\mu \subset T$ such that for any $\alpha 
\in \mu$ bounding distinct triangles $\Delta_1$ and $\Delta_2$, we have $\Delta_1\cap \Delta_2 \cap \mu = \alpha$. 
A \emph{simultaneous flip} sends $T$ to a triangulation $T'$ by replacing each $\alpha \in \mu$ by the other diagonal of its bounding quadrilateral.
\end{definition}

Figure \ref{fig:sim-flip-example} depicts an example of a simultaneous flip. We can now define the \emph{flip graph} $\cF(\Sigma)$ of a surface $\Sigma$, whose vertices are isotopy classes of ideal triangulations of $\Sigma$. If $\Sigma$ is an infinite-type surface, we connect two triangulations by an edge if they are related by a simultaneous flip. If $\Sigma$ is finite-type, we connect triangulations by an edge only if they are related by a single flip.

\begin{figure}[htbp]
    \centering
    \includegraphics[width=0.3\textwidth]{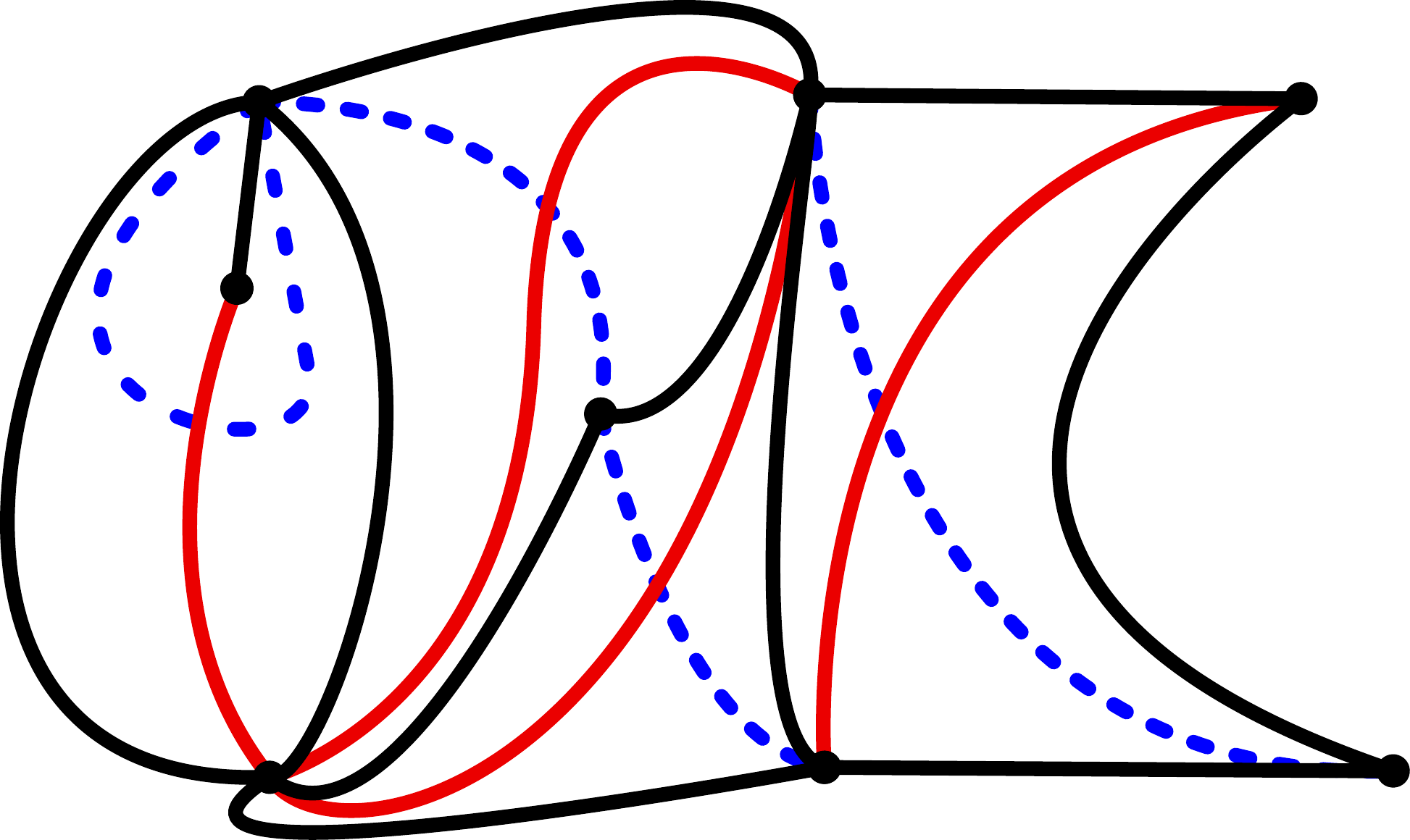}
    \caption{An example of a simultaneous flip of a triangulation, with the dotted blue arcs before the flip and the red arcs afterward.}
    \label{fig:sim-flip-example}
\end{figure}

Although the flip graph is connected in the finite-type case \cite{disarlo-parlier}, the flip graph in the infinite-type case has uncountably many connected components. In particular, \cite{fossas-parlier} establishes conditions on when two triangulations are in the same connected component.

\begin{lemma}[Theorem 1.1 of \cite{fossas-parlier}]\label{lemma:finite-int-same-cc}
    Let $\Sigma$ be an infinite-type surface. Let $S$ and $T$ be triangulations of $\Sigma$. Then, $S$ and $T$ are in the same connected component of $\cF(\Sigma)$ if and only if there exists $K \geq 0$ such that for every arc $\alpha$ of $S$ and every arc $\beta$ of $T$, the intersection numbers $i(\alpha, T)$ and $i(\alpha, S)$ are bounded by $K$.
\end{lemma}

Note that there is a natural action of the mapping class group on the flip graph sending an ideal triangulation to its image under the homeomorphism of the surface. 
Since this action commutes with flips, each mapping class induces an automorphism of the flip graph. In the next section, we show that there exists an automorphism of the flip graph which is not induced by any mapping class.

\section{Proof of Theorem \ref{thm:isomorphic-to-subgroup}}
In this section, we show that the extended mapping class group 
is a strict subgroup of the automorphism group of the flip graph for infinite-type surfaces. Our methods were first inspired by those used by Branman to show an analogous result for pants graphs \cite{branman}.

Since the flip graph is disconnected, the connected components can in some sense be manipulated independently of one another. However, every mapping class necessarily affects triangulations in every connected component of the flip graph. This is the essential reason that the symmetries of $\Sigma$ fail to fully capture the symmetries of $\cF(\Sigma)$. After establishing some properties of the action on $\cF(\Sigma)$, we will prove this by providing an automorphism of $\cF(\Sigma)$ that cannot be induced by a mapping class.

We begin by proving that mapping classes are determined by their action on certain arcs.
\alex
\begin{proof}
Assuming (ii), we see that restricting the isotopy from $\Id_\Sigma$ to $f$ yields an isotopy from $\lambda$ to $f(\lambda)$. Hence, (ii) implies (i). To show the converse, we claim that any mapping class satisfying (i) must fix every isotopy class of curves on $\Sigma$. Then we can use the Alexander Method of \cite{alex-ref-pls} to conclude that $f$ must be isotopic to $\Id_\Sigma$.

\begin{figure}[ht]
    \centering
    \begin{subfigure}[t]{0.25\textwidth}
        \centering
        \includegraphics[width=0.9\textwidth]{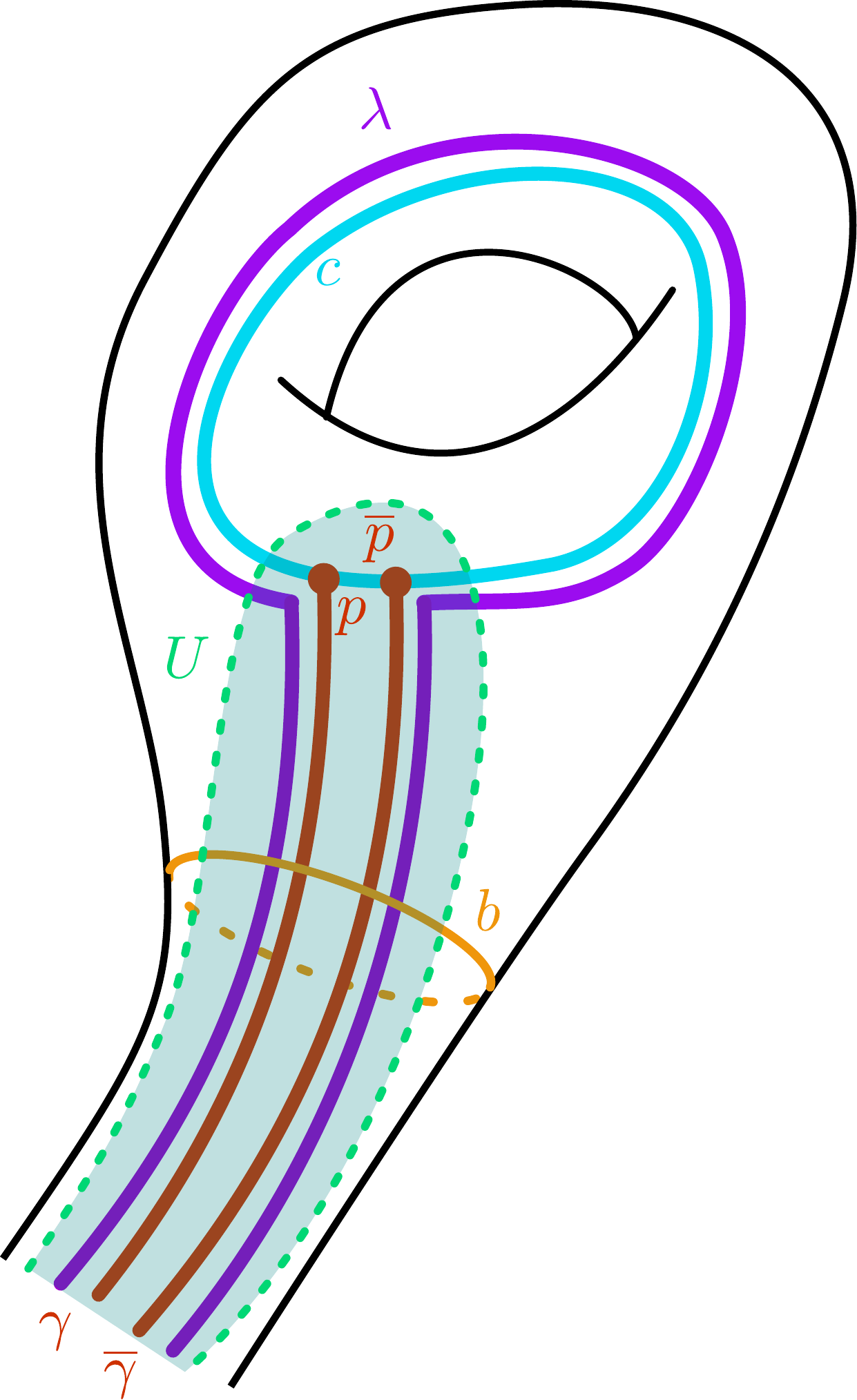}
    \end{subfigure}
    ~ 
    \begin{subfigure}[t]{0.25\textwidth}
        \centering
        \includegraphics[width=0.9\textwidth]{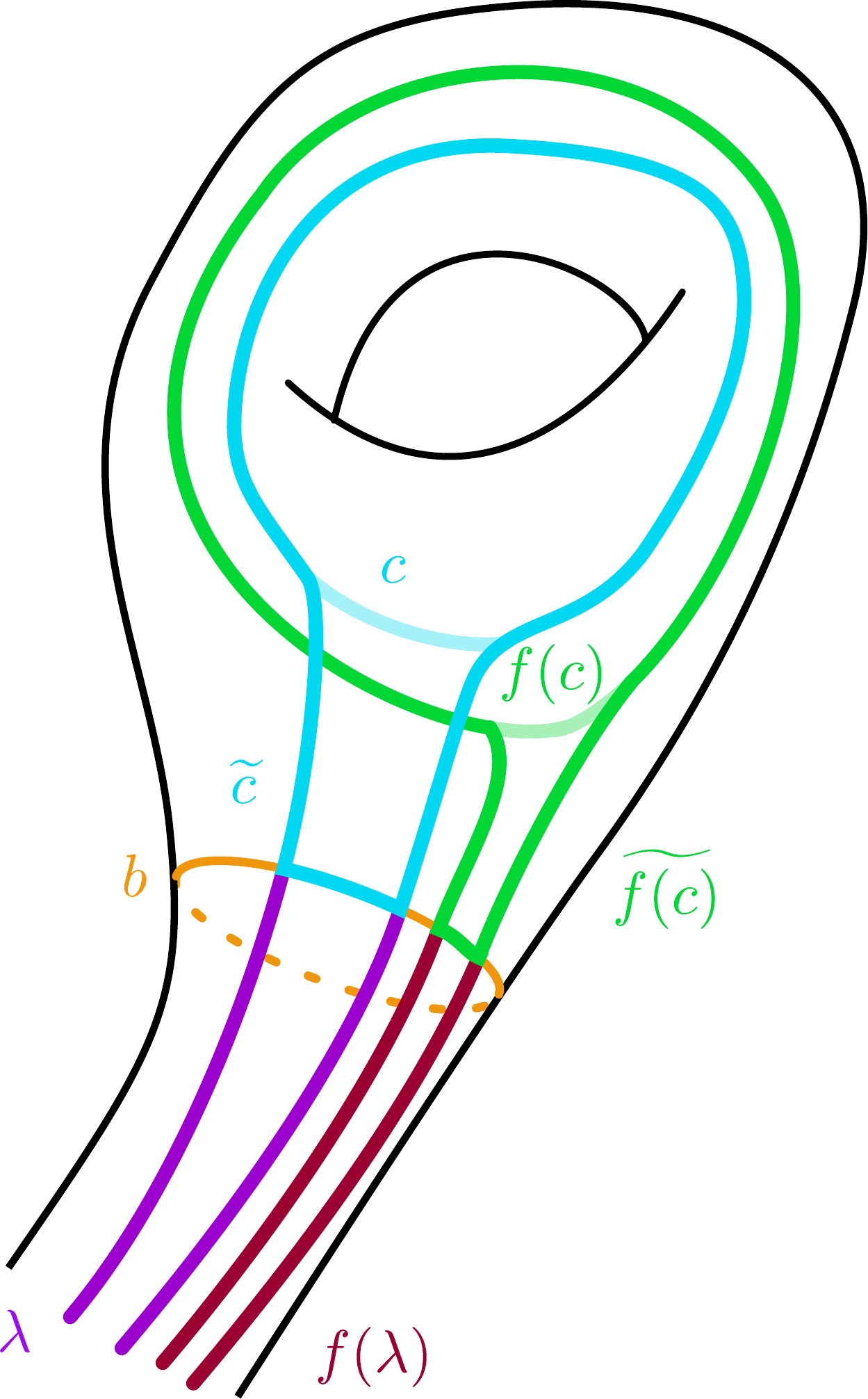}

    \end{subfigure}
        \caption{A diagram showing the main ideas in Lemma \ref{lemma:fix-arc-fix-curve-alex-ref-pls}.}
        \label{fig:alex-method}
\end{figure}

Let $c$ be a curve in $\Sigma$, and let $\gamma$ be a path from $e$ to a point $p$ on $c$ such that $|c \cap \gamma| = 1$.  Choose a tubular neighborhood $U$ of $\gamma$ such that $|\partial U \cap c| = 2$. Let $\overline{\gamma}$ be a path in $U$ from a point $\overline{p}$ on $c$ to $e$ such that $\gamma \cap \overline{\gamma} = \emptyset$ and $|c \cap \overline{\gamma}| = 1$. Define $\lambda \in \mathcal{L}_e$ to be the arc that begins at $e$, follows $\gamma$ from $e$ to $p$, then follows the connected component of $c\setminus \{p,\overline{p}\}$ not contained in $U$, and finally runs back to $e$ along $\overline{\gamma}$. See Figure \ref{fig:alex-method}. Since $c \cup f(c)$ is compact, we can choose a separating curve $b$ such that one component of $\Sigma \setminus b$ contains $e$ and the other contains $c \cup f(c)$. We also require that $b$ intersects $\gamma$, $\overline{\gamma}$, $f(\gamma)$, and $f(\overline{\gamma})$ at exactly one point each.

We now define an isotopy from $c$ to $f(c)$ in terms of the isotopy $I$ from $\lambda$ to $f(\lambda)$ given by the hypothesis. First, observe that $c$ is isotopic to a curve $\widetilde{c}$ contained in $\lambda \cup b$. Likewise, $f(c)$ is isotopic to a curve $\widetilde{f(c)}$ contained in $f(\lambda) \cup b$. $I$ extends to an isotopy between $\widetilde{c}$ and $\widetilde{f(c)}$ by translating $\widetilde{c} \cap b$ along $b$.
\end{proof}

\begin{figure}[ht]
    \centering
    \includegraphics[width=.3\textwidth]{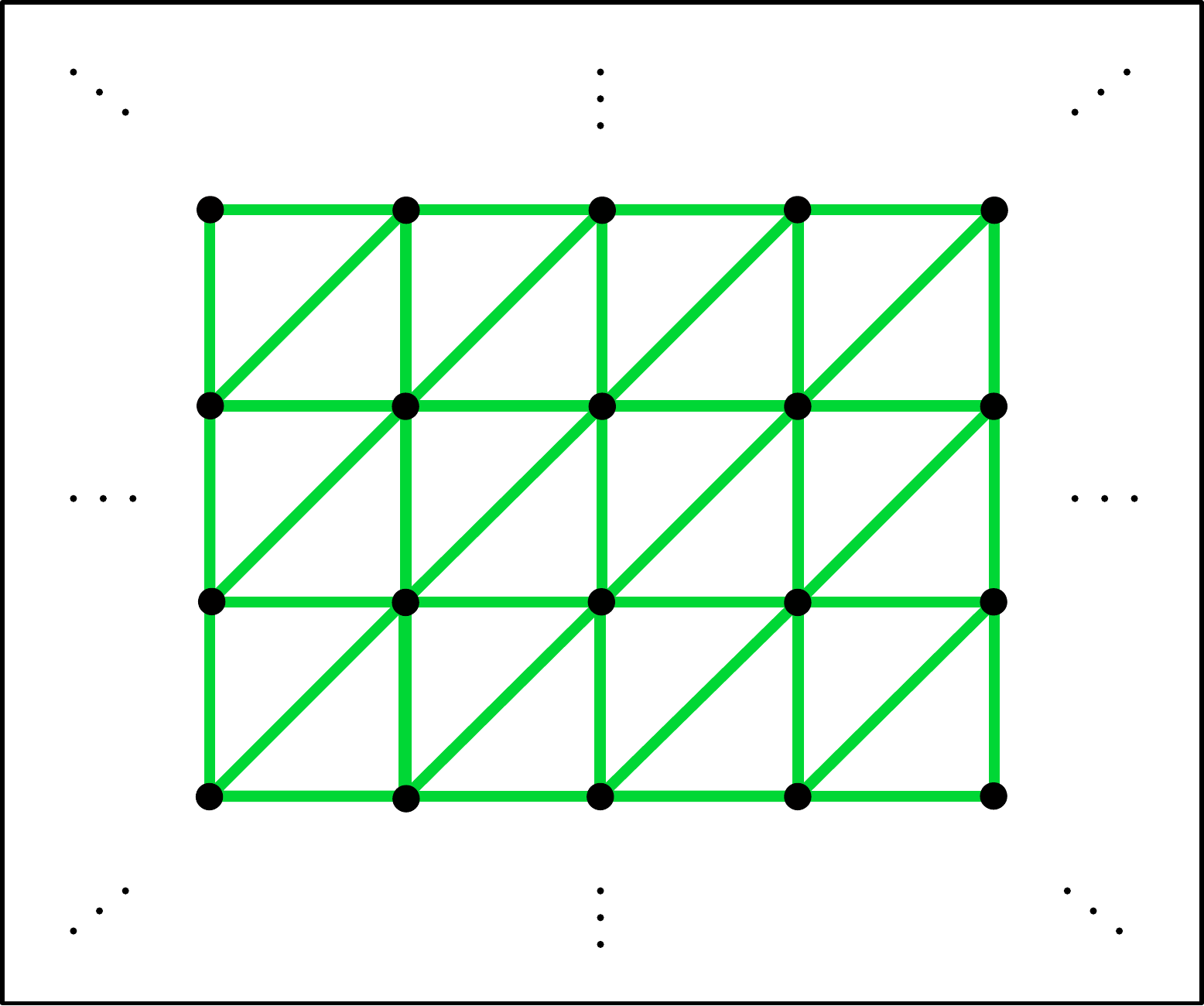}
    \caption{The standard triangulation of the flute surface.}
    \label{fig:flute_triangulations}
\end{figure}

\begin{remark}
Initially, one might think that if a mapping class fixes a particular triangulation, it must be the identity. This is not always the case, however. For example, consider the triangulation $T$ of the flute surface depicted in Figure \ref{fig:flute_triangulations}. Then, the mapping class that shifts every point on the plane to the right one unit fixes $T$ (not as a set, but as a triangulation). Note that this homeomorphism is not isotopic to the identity. In other words, arbitrary triangulations do not form Alexander systems. This implies that generally, $\EMod(\Sigma)$ does not act freely on $\cF(\Sigma)$.
\end{remark}

The following technical lemma will be used in the proof of Lemma \ref{lemma:no-map-one-comp-nontrivial} to glean information about the action of $\cF(\Sigma)$ on individual arcs from its action on triangulations. 
\begin{lemma}\label{complete-one-but-not-other} 

Let $\alpha$ and $\beta$ be distinct ideal arcs on $\Sigma$, such that $\alpha$ is not a petal arc for $\beta$. Then, there exist ideal triangulations $T_1$, $T_2$ in distinct connected components of the flip graph that each contain $\alpha$ but not $\beta$.
\end{lemma}
\begin{proof}
We first construct $T_1$. If $i(\alpha, \beta) > 0$, then any completion of $\alpha$ can be taken as $T_1$. Otherwise, by Lemma \ref{lemma:complete-multiarc-to-tri}, arcs $\alpha$ and $\beta$ can be completed into a triangulation $T$. If $\beta$ is flippable, we flip $\beta$ to get a triangulation $T'$ that contains $\alpha$ but not $\beta$.
Otherwise, $\beta$ is unflippable, which implies that the triangles bordering both sides of the arcs are the same, and that $\beta$ is contained within a petal formed by an arc $\gamma$. Then, $\gamma$ must border two distinct triangles, so $\gamma$ is flippable. After flipping $\gamma$, $\beta$ will thus be flippable, and flipping $\beta$ will result in a triangulation $T_1$ that contains $\alpha$ but not $\beta$.

Next, we construct $T_2$. Since $\Sigma' = \Sigma \setminus (\alpha \cup \beta)$ is still an infinite-type surface, we can choose a countable collection of disjoint simple closed curves $\delta_i$ on $\Sigma'$. Note that each $\delta_i$, when thought of as a curve on $\Sigma$, is disjoint from $\alpha$ and $\beta$. 
We apply the procedure described in Corollary 4.1 of \cite{fossas-parlier} to $T_1$ using the collection $\{\delta_i\}$ to obtain a triangulation $T_2$ in a distinct connected component of $\cF(\Sigma)$, containing $\alpha$ but not $\beta$.
\end{proof}

\begin{remark}
Note that the procedure described above can create uncountably many triangulations, each in a distinct connected component and containing $\alpha$ but not $\beta$. This procedure is done precisely in \cite{fossas-parlier}. However, it is not true that every connected component contains such a triangulation. Consider the triangulation $T$ of the flute surface shown in Figure \ref{fig:flute_triangulations}. We claim that no triangulation in the connected component of $T$ contains an arc $\alpha$ with an endpoint at the non-isolated end. To see why this is true, observe that none of the arcs of $T$ connect to the non-isolated end, so this end cannot be the vertex of a bounding quadrilateral of an arc. Hence, a simultaneous flip cannot connect any new arcs to this end. Since the resulting triangulation will retain this property, we conclude that no triangulation containing $\alpha$ is reachable from $T$ by a series of simultaneous flips. This implies that the modular flip graph \cite{disarlo-parlier} is disconnected for certain infinite-type surfaces, a result which is of interest of its own right.
\end{remark}

The next result shows that mapping classes are determined by their action on the flip graph. In addition, the weaker hypothesis allows us to determine the action of a mapping class on one connected component of $\cF(\Sigma)$ by its action on the other components. We use this in the proof of Theorem \ref{thm:isomorphic-to-subgroup} to show that a particular automorphism of $\cF(\Sigma)$ is not induced by a mapping class.
\begin{lemma}\label{lemma:no-map-one-comp-nontrivial}
If $f \in \EMod(\Sigma)$ acts trivially on all except for one connected component of $\cF(\Sigma)$, then $f$ is the identity in $\EMod(\Sigma)$.
\end{lemma}

\begin{proof}
Let $\alpha$ an arc in $\Sigma$. We show that $f(\alpha)$ cannot be $\beta$ for any arc $\beta \neq \alpha$, which implies that $f(\alpha) = \alpha$. Let $\beta \neq \alpha$ be arbitrary.

First, suppose that $\alpha$ is a petal arc of $\beta$. We observe that $\beta$ must have two distinct endpoints, while $\alpha$ only has one, so $f(\alpha)$ cannot be $\beta$. In the case where $\alpha$ is not a petal arc of $\beta$, Lemma \ref{complete-one-but-not-other} guarantees the existence of  triangulations $T_1$, $T_2$ containing $\alpha$ but not $\beta$ such that $T_1$ and $T_2$ are in distinct connected components of $\cF(\Sigma)$. Then, $f$ must act trivially on either $T_1$ or $T_2$, so $f(\alpha) \neq \beta$. Thus, $f(\alpha) = \alpha$, and by Lemma \ref{lemma:fix-arc-fix-curve-alex-ref-pls}, $f$ is the identity in $\EMod(\Sigma)$.
\end{proof}

\begin{corollary}\label{corollary:faithful-action}
The action of $\EMod(\Sigma)$ on $\cF(\Sigma)$ is faithful. In other words, $\EMod(\Sigma)$ embeds into $\Aut(\cF(\Sigma)).$
\end{corollary}

The preceding result establishes that the flip graph has at least as many symmetries as the surface. But since so little is known about the graph structure of $\cF(\Sigma)$, it is difficult to conceive of automorphisms that are not induced by specific mapping classes. The following lemma improves the situation to some extent, allowing us to construct automorphisms of individual connected components via compactly-supported mapping classes.
\begin{lemma}\label{lemma:comp-supp-aut}
Let $f$ be a compactly-supported mapping class. Then, $f$ induces an automorphism on every connected component of $\cF(\Sigma)$. In particular, 
$f$ fixes the connected components of $\cF(\Sigma)$.
\end{lemma}
\begin{proof}
Let $T \in \cF(\Sigma)$ be arbitrary. We will show that $T$ and $f(T)$ must lie in the same connected component of $\cF(\Sigma)$. Since $f$ has compact support, there exists a finite-type subsurface $\Sigma'$ of $\Sigma$ containing every point $x 
\in T \subset \Sigma$ for which $f(x) \neq x$. It follows that every transverse intersection of arcs of $T$ and arcs of $f(T)$ is contained within $\Sigma'$.

Only finitely many arcs of $T$ and $f(T)$ can pass through $\Sigma'$, because 
otherwise, the arcs of $T$ or $f(T)$ would accumulate. Let 
$\mc{A} \subset \Sigma'$ be this set of arcs, and let $K = \max_{\alpha,\beta\in\mc{A}} |\alpha\cap \beta|$. This maximum exists, because intersection points between arcs in $\mc{A}$ are discrete and contained in $\Sigma'$. Note that for any $\alpha \in T$, we have $i(\alpha,f(T)) \le K |\mc{A}|$, and for any $\beta \in f(T)$, we have $i(\beta,T) \le K|\mc{A}|$. Thus, by Lemma \ref{lemma:finite-int-same-cc}, $T$ and $f(T)$ lie in the same connected component of $\cF(\Sigma)$.

Since $f$ induces an automorphism on $\cF(\Sigma)$, it must restrict to an automorphism on each connected component of $\cF(\Sigma)$.
\end{proof}

We are now ready to prove Theorem \ref{thm:isomorphic-to-subgroup}.

\begin{proof}[Proof of Theorem \ref{thm:isomorphic-to-subgroup}]
Let $f$ be a nontrivial mapping class with compact support, and let $\Gamma$ be a connected component of $\cF(\Sigma)$ that $f$ acts nontrivially on. Note that $f$ induces an automorphism on $\Gamma$ by Lemma \ref{lemma:comp-supp-aut}, which we also denote $f$. We define an automorphism $\overline{f}$ of $\cF(\Sigma)$ by
\[\overline{f}(T) = \begin{cases} 
    f(T) & T \in \Gamma \\
    T & T \in \cF(\Sigma) \setminus \Gamma
                    \end{cases}.\]
By Lemma \ref{lemma:no-map-one-comp-nontrivial}, a mapping class inducing $\overline{f}$ would necessarily act trivially on $\cF(\Sigma)$, which $\overline{f}$ does not. Thus, $\overline{f}$ cannot be induced by a mapping class. Together with Corollary \ref{corollary:faithful-action}, this establishes that the natural action of $\EMod(\Sigma)$ corresponds to a homomorphism $\EMod(\Sigma) \to \Aut(\cF(\Sigma))$ that is injective but not surjective.
\end{proof}

\section{Discussion}
In this paper, we constructed an automorphism of the flip graph not induced by a mapping class. In other words, there are graph isomorphisms between flip graphs that are not induced by homeomorphisms of the corresponding surfaces. A natural follow-up question to ask is whether simultaneous flip graphs even determine a surface:

\begin{question}
If $\cF(\Sigma)$ and $\cF(\Sigma')$ are isomorphic as graphs, does this mean that $\Sigma$ and 
$\Sigma'$ are homeomorphic?
\end{question}

The method we used in this paper considers how automorphisms of the flip graph on distinct connected components may be incompatible with mapping classes. However, we might try to restrict our attention to a single connected component of the flip graph. We can ask whether every automorphism of a connected component is induced by a mapping class:

\begin{question}
For a given connected component $\Gamma$ of the flip graph $\cF(\Sigma)$, is $\Aut(\Gamma)$ isomorphic to
$\EMod(\Sigma)$?
\end{question}

\bibliography{FlipGraphsIvanov}
\end{document}